\documentclass{article}
\usepackage{cite}
\usepackage{amsmath,amsthm,amsfonts,amssymb}
\usepackage{color}

\usepackage{ifthen}
\allowdisplaybreaks
\newtheorem{thm}{Theorem}[section]
\newtheorem{cor}[thm]{Corollary}
\newtheorem{lem}[thm]{Lemma}
\newtheorem{prop}[thm]{Proposition}
\newtheorem*{lemalt}{Lemma \ref{addexp1}$'\!\!$}
\theoremstyle{remark}
\newtheorem{rem}[thm]{Remark}
\numberwithin{equation}{section}

\newcommand{\parder}[3][Default]{
    \frac{\partial \ifthenelse{\equal{#1}{Default}}{}{^{#1}}#2}{
              \partial #3 \ifthenelse{\equal{#1}{Default}}{}{^{#1}}}}
\newcommand{\dder}[3][Default]{
    \frac{\mathrm{d} \ifthenelse{\equal{#1}{Default}}{}{^{#1}}#2}{
              \mathrm{d} #3 \ifthenelse{\equal{#1}{Default}}{}{^{#1}}}}
\newcommand{\jac}{{\mathcal J}}

\newcommand{\Mat}{\operatorname{Mat}}

\newcommand{\rk}{\operatorname{rk}}
\newcommand{\imp}{{\mathversion{bold}$\Rightarrow$} }

\newcommand{\C}{{\mathbb C}}

\newcommand{\N}{{\mathbb N}}

\newcommand{\tp}{{^{\mathrm t}}}

\newcommand{\subs}[3]{#1|_{\ifthenelse{\equal{#2}{X}}{}{#2=}#3}}
\newcommand{\subss}[3]{(#1)|_{\ifthenelse{\equal{#2}{X}}{}{#2=}#3}}
\newcommand{\bsubs}[4]{#2#1|_{\ifthenelse{\equal{#3}{X}}{}{#3=}#4}}
\newcommand{\bsubss}[4]{#1(#2#1)#1|_{\ifthenelse{\equal{#3}{X}}{}{#3=}#4}}
\newcommand{\bbsubs}[5]{#3#1|#2_{\ifthenelse{\equal{#4}{X}}{}{#4=}#5}}
\newcommand{\bbsubss}[5]{#1(#3#1)#1|#2_{\ifthenelse{\equal{#4}{X}}{}{#4=}#5}}
\newcommand{\wtilde}[1]{\tilde{#1}}

\begin{document}

\title{Polynomial maps with invertible sums of Jacobian matrices and of directional
derivatives\footnote{Supported by
NSF of China (No.11071097, No.11026039) and ``211 Project" and ``985 Project" of Jilin University}}
\author{Hongbo Guo, Michiel de Bondt,
Xiankun Du\footnote{Corresponding author}, Xiaosong Sun\\
School of Mathematics, Jilin University, Changchun 130012, China\\
Radboud University, Nijmegen, The Netherlands\footnote{Institute of second author}\\
Email: 0417ghb@163.com, debondt@math.ru.nl, duxk@jlu.edu.cn, \\ sunxs@jluedu.cn}

\maketitle

\begin{abstract} \noindent
Let $F: \mathbb{C}^n\rightarrow
 \mathbb{C}^m$ be a polynomial map with $\deg
F=d\geq 2$. We prove
that $F$ is invertible if $m = n$ and
$\sum^{d-1}_{i=1}\subss{\jac F}{X}{\alpha_i}$ is invertible for all
$i$, which is trivially the case for invertible quadratic maps.

More generally, we prove that for affine lines
$L = \{\beta + \mu \gamma \mid \mu \in \C\} \subseteq \C^n$ ($\gamma \ne 0$),
$F|_L$ is linearly rectifiable, if and only if
$\sum^{d-1}_{i=1}\subss{\jac F}{X}{\alpha_i} \cdot \gamma \ne 0$
for all $\alpha_i\in L$. This appears to be the case for
all affine lines $L$ when $F$ is injective and $d \le 3$.

We also prove that if $m = n$ and $\sum^{n}_{i=1}\subss{\jac F}{X}{\alpha_i}$ is invertible
for all $\alpha_i\in \mathbb{C}^n$, then $F$ is a composition of an invertible linear map
and an invertible polynomial map $X+H$ with linear part $X$,
such that the subspace generated by $\{\subss{\jac H}{X}{\alpha} \mid
\alpha\in\mathbb{C}^n\}$ consists of nilpotent matrices.
\end{abstract}



\section{Introduction}

Denote by $\jac F$ the Jacobian matrix of a  polynomial map
$F:\mathbb{C}^n\rightarrow\mathbb{C}^n$.  The Jacobian conjecture
 states that $F$ is invertible if $\jac F$ is
invertible, or  equivalently if $\subss{\jac F}{X}{\alpha}$ is
invertible for all $\alpha \in \mathbb{C}^n$. The conjecture has been
reduced to polynomial maps
of the form $F=X+H$, where $H$ is homogeneous (of degree 3) and $\jac H$ is
nilpotent, by Bass, Connell and Wright in \cite{Bass}, and independently by
Yagzhev in \cite{Yagzhev}. Subsequent reductions are to the case where for
the polynomial map $F = X + H$ above, each component of $H$ is a cube of
a linear form, by Dru\.{z}kowski in \cite{Druzkowski}, and to the
case where $\jac H$ is symmetric, by De Bondt and Van den Essen in
\cite{Bondt}, but these reductions cannot be applied
simultaneously, see also \cite{BondtEssen}. 
More details about the Jacobian conjecture can be found
in \cite{Arno} and \cite{homokema}.

Invertibility of a polynomial map $F$ has been examined by several authors under certain conditions on the evaluated Jacobian matrices $\subss{\jac F}{X}{\alpha}, ~\alpha\in\mathbb{C}^n$. With an extra assumption
that $F-X$ is cubic homogeneous, Yagzhev proved in
 \cite{Yagzhev} that if $\subss{\jac F}{X}{\alpha_1}+\subss{\jac F}{X}{\alpha_2}$ is
  invertible for all $\alpha_1,\alpha_2\in \mathbb{C}^n$, then the polynomial map $F$
  is invertible. The Jacobian matrix $\jac H$ of a polynomial map $H$ is called strongly nilpotent if $\subss{\jac H}{X}{\alpha_1}\cdot\subss{\jac H}{X}{\alpha_2}\cdot\cdots\cdot
  \subss{\jac H}{X}{\alpha_n}=0$ for all $\alpha_i \in \mathbb{C}^n$.
  Van den Essen and Hubbers proved in \cite{Essen} that  $\jac H$ is strongly nilpotent if and only if there exists
$T\in GL_n(\mathbb{C})$ such that $T^{-1}\jac(H)T$ is strictly upper triangular, if and only if the
 polynomial map $F=X+H$ is linearly triangularizable (so $F$ is invertible).
 This result was generalized by Yu in \cite{Yu}, where he additionally observed
that $\jac H$ is already strongly nilpotent if
$\subss{\jac H}{X}{\alpha_1}\cdot\subss{\jac H}{X}{\alpha_2}\cdot\cdots\cdot
\subss{\jac H}{X}{\alpha_m}=0$ for some $m \in \N$.

In \cite{Sun}, Sun extended the notion of strong nilpotency and proved that a polynomial map $F=X+H$ is invertible if the Jacobian matrix $\jac H$ is \emph{additive-nilpotent}, i.e.\@ $\sum_{i=1}^m \subss{\jac H}{X}{\alpha_i}$ is nilpotent for each positive integer $m$
 and all $\alpha_i\in \mathbb{C}^n$,
 which generalizes results in \cite{Essen,wang,Yagzhev,Yu}.
Instead of looking at polynomial maps $F=X+H$ such that $\jac H$ is nilpotent, we look
at polynomial maps $F$ in general, and assume that
$\det \sum_{i=1}^{d-1} \subss{\jac F}{X}{\alpha_i} \ne 0$ for all
$\alpha_i\in \mathbb{C}^n$, where $d = \deg F$.
More generally, we only assume that $\sum_{i=1}^{d-1} \subss{\jac F}{X}{\alpha_i}
\cdot \gamma \ne 0$
and only for $\alpha_i\in \mathbb{C}^n$ which are collinear, where $\gamma \ne 0$ is
the direction of the line.

Observe that if $F=X+H$ is a polynomial map such that $\jac H$ is
additive-nilpotent, then $\sum_{i=1}^{m}\subss{\jac\wtilde{F}}{X}{\alpha_i}$ is
invertible for all $m \in \N$ and all $\alpha_i\in \mathbb{C}^n$,
where $\wtilde{F}=L_1\circ F\circ L_2$ is a composition of
$F$ and invertible linear maps $L_1$ and $L_2$. Conversely, it is interesting to
describe the polynomial maps such that sums of the evaluated
Jacobian matrices are invertible. In this paper, we first prove that
a polynomial map $F$ of degree $d$ is invertible if
$\sum_{i=1}^{d-1} \subss{\jac F}{X}{\alpha_i}$ is invertible for all $\alpha_i\in
\mathbb{C}^n$. This generalizes results of Wang in \cite{wang}, Yagzhev in
\cite{Yagzhev}, Van den Essen
  and Hubbers in \cite{Essen} and Sun in \cite{Sun}. Then we prove the invertibility of a polynomial map $F$ such that $\sum_{i=1}^{n}\subss{\jac F}{X}{\alpha_i}$ is invertible for all $\alpha_i\in \mathbb{C}^n$, and finally characterize such a polynomial map as a composition of  an invertible linear map and an invertible polynomial map $X+H$ such that $\jac H$ is additive-nilpotent.

\section{Additive properties of the derivative on lines}

\begin{lem} \label{addexp1}
Assume $\lambda_1 ,\lambda_2, \ldots, \lambda_{d-1} \in \C$ such that
$\sum_{i\in I} \lambda_i \ne 0$ for all nonempty $I \subseteq \{1,2,\ldots,d-1\}$,
and $P \in \C[[T]]$ with constant term $\lambda_1 + \lambda_2 + \cdots
+ \lambda_{d-1}$. Then there are $r_1, r_2, \ldots, \allowbreak r_{d-1}
\in \C$ such that
$$
P - \sum_{i=1}^{d-1} \lambda_i \exp(r_i T)
$$
is divisible by $T^d$, where $\exp(T) = \sum_{j=0}^{\infty}
\frac{1}{j!} T^j$.
\end{lem}

\begin{proof}
Write
$$
P = \sum_{j=0}^{\infty} \frac{p_j}{j!} T^i
$$
Then we must find a solution $(Y_1,Y_2,\ldots, Y_{d-1}) = (r_1,r_2,\ldots,r_{d-1})
\in \C^{d-1}$ of
\begin{equation}
\sum_{i=1}^{d-1} \lambda_i Y_i^j = p_j \quad (j = 0,1,\ldots,d-1) \label{eq1}
\end{equation}
The equation for $j = 0$ is fulfilled by assumption, and finding a solution of (\ref{eq1})
is the same as finding a solution $(Y_1,Y_2,\ldots, Y_d) = (r_1,r_2,\ldots,r_d)$ of
\begin{equation}
\sum_{i=1}^{d-1} \lambda_i Y_i^j = p_j Y_d^j \quad (j = 1,\ldots,d-1) \label{eq2}
\end{equation}
for which $r_d = 1$. Since $(Y_1,Y_2,\ldots, Y_d) = 0$ is a solution of (\ref{eq2}),
it follows from Krull's Height Theorem that the dimension of the set of solutions $(r_1,r_2,\ldots,r_d) \in \C^d$ of (\ref{eq2}) is at least one. Hence
there exists a nonzero solution $(r_1,r_2,\ldots,r_d) \in \C^d$ of (\ref{eq2}).

If $r_d \ne 0$, then $r_d^{-1}(r_1,r_2,\ldots,r_d)$ is a solution of (\ref{eq2})
as well, because the equations of (\ref{eq2}) are homogeneous. Hence
$r_d^{-1}(r_1,r_2,\ldots,r_{d-1})$ is a solution of (\ref{eq1}) in that case.
So assume that $r_d = 0$. Then $\sum_{i=1}^{d-1} \lambda_i r_i^j = 0$ for all $j$.
Take $e \le d-1$ and nonzero
$s_1 < s_2 < \cdots < s_e$ such that $\{0,r_1, r_2, \ldots, r_{d-1}\} = \{0,s_1, s_2,
\ldots,s_e\}$. Then $e \ge 1$ because $(r_1,r_2,\ldots,r_d) \ne 0$, and
$$
0 = \sum_{i=1}^{d-1} \lambda_i r_i^j
= \sum_{k=1}^e s_k^j \sum_{r_i = s_k} \lambda_i
$$
for all $j$ such that $1 \le j \le e$. This means that the vector $v$ defined by
$v_k := \sum_{r_i = s_k} \lambda_i$ for all $k$ satisfies
$M v = 0$, where $M$ is the Vandermonde matrix with entries $M_{jk} = s_k^j$.
Since $v_k$ is nonzero by assumption for all $k$, this
contradicts $\det M \ne 0$.
\end{proof}

\noindent
Let $f \in \C[X] = \C[X_1,X_2,\ldots, X_n]$ be a polynomial of degree $d$
and $\beta, \gamma \in \C^n$. Set $g(T) := f(\beta + T \gamma)$ and
$D := \sum_{i=1}^n \gamma_i \parder{}{X_i}$. Notice that $T \mapsto D$
induces an isomorphism of $\C[T]$ and $\C[D]$. By the chain rule,
\begin{align*}
\dder[i]{}{T} \big(f(\beta + T \gamma)\big) &= \dder[i-1]{}{T}
\big(\subss{\jac f}{X}{\beta + T \gamma} \cdot \gamma\big) \\
&= \dder[i-1]{}{T} \big((D f) (\beta + T \gamma)\big) = (D^i f) (\beta + T \gamma)
\end{align*}
follows for all $i \in \N$ by induction on $i$.
Using the Taylor series at $0$ of $g$, we see that for all $c \in \C$,
\begin{align}
f(\beta + c \gamma) = g(c) &= \sum_{i=0}^{\infty} \frac{(c-0)^i}{i!}
\bsubss{\bigg}{\dder[i]{}{T} g(T) }{T}{0} \nonumber \\
&= \sum_{i=0}^{\infty} \frac{c^i}{i!}
\bsubss{\bigg}{\dder[i]{}{T} f(\beta + T \gamma) }{T}{0} \nonumber \\
&= \bsubs{\bigg}{\sum_{i=0}^{\infty} \frac{c^i}{i!}
\big((D^i f)(\beta + T \gamma) \big)}{T}{0} \nonumber \\
&= \bsubs{\Big}{\bsubss{\big}{(\exp c D) f}{X}{\beta + T \gamma}}{T}{0}
= \bsubss{\big}{(\exp c D) f}{X}{\beta} \label{taylor}
\end{align}

\begin{prop} \label{lineinj}
Let $F: \C^n \rightarrow \C^m$ be a polynomial map of degree $d$ and
$\lambda_i \in \C$ for all $i$, such that
$\sum_{i\in I} \lambda_i \ne 0$ for all nonempty
$I \subseteq \{1,2,\ldots,d-1\}$.
Assume $\beta, \gamma \in \C^n$ such that $\gamma \ne 0$.
If every sum of $d-1$ directional derivatives of $F|_{\beta + \C \gamma}$
along $\gamma$ is nonzero ($\lambda_i = 1$ for all $i$ below), or more
generally,
$$
\sum_{i=1}^{d-1} \lambda_i \cdot \subss{\jac F}{X}{\alpha_i} \cdot \gamma \ne 0
$$
for all $\alpha_i \in \{ \beta + \mu \gamma \mid \mu \in \C\}$, then
$F(\beta) \ne F(\beta + \gamma)$.
\end{prop}

\begin{proof}
Set $D := \sum_{i=1}^n \gamma_i \parder{}{X_i}$ and $P(T) := \big(\sum_{i=1}^{d-1}
\lambda_i\big) T^{-1}(\exp(T)-1)$. By (\ref{taylor}),
\begin{align*}
\bigg(\sum_{i=1}^{d-1} \lambda_i\bigg) \cdot
   \big(F_j(\beta + \gamma) - F_j(\beta)\big)
&= \bigg(\sum_{i=1}^{d-1} \lambda_i\bigg) \cdot
   \bsubss{\Big}{\big(\exp(D) - 1\big) F_j}{X}{\beta} \\
&= \bsubss{\big}{ D P(D) F_j}{X}{\beta} = \bsubss{\big}{ P(D) (D F_j)}{X}{\beta}
\end{align*}
for all $j$. Choose $r_i$ as in Lemma \ref{addexp1} for all $i$.
From the definition of $D$ and (\ref{taylor}) with $c = r_i$ and $f = D F_j$,
\begin{align*}
\sum_{i=1}^{d-1} \lambda_i \cdot \subss{\jac F_j}{X}{\beta + r_i \gamma} \cdot \gamma
&= \sum_{i=1}^{d-1} \lambda_i(D F_j) (\beta + r_i \gamma) \\
&= \bsubss{\bigg}{\sum_{i=1}^{d-1} \lambda_i \exp(r_i D) (D F_j) }{X}{\beta}
\end{align*}
follows for all $j$. Since $P(T) - \sum_{i=1}^{d-1} \lambda_i \exp(r_i T)$
is divisible by $t^d$ and $D F_j$ has degree at most $d - 1$, we have
$$
P(D) (D F_j) =
\sum_{i=1}^{d-1} \lambda_i \exp(r_i D) (D F_j)
$$
for all $j$. By substituting $x = \beta$ on both sides, we obtain
$$
\bigg(\sum_{i=1}^{d-1} \lambda_i\bigg) \cdot \big(F_j(\beta + \gamma) - F_j(\beta)\big) =
\sum_{i=1}^{d-1} \lambda_i \subss{\jac F_j}{X}{\beta+r_i \gamma} \cdot \gamma
$$
for all $j$, which gives the desired result.
\end{proof}

\begin{cor} \label{allinj}
Let $F: \C^n \rightarrow \C^m$ be a polynomial map of degree $d$ and
$\lambda_i \in \C$ for all $i$, such that $\sum_{i\in I} \lambda_i \ne 0$
for all nonempty $I \subseteq \{1,2,\ldots,d-1\}$. If $\rk (\sum_{i=1}^{d-1}
\lambda_i \subss{\jac F}{X}{\alpha_i}) = n$ for all $\alpha_i \in \C^n$, then
$F$ is injective.

If additionally $n = m$, then $F$ is an invertible polynomial map.
\end{cor}

\begin{proof}
Assume $F(\beta) = F(\beta + \gamma)$ for some $\beta, \gamma \in \C^n$
By Proposition \ref{lineinj}, there are $\alpha_i \in \C^n$ such that
$$
\sum_{i=1}^{d-1} \lambda_i \cdot \subss{\jac F}{X}{\alpha_i} \cdot \gamma = 0
$$
and in particular
$\rk \big(\sum_{i=1}^{d-1} \lambda_i \cdot \subss{\jac F}{X}{\alpha_i}\big) \ne n$.

If $n = m$, then a special case of the Cynk-Rusek Theorem in \cite{Cynk} (see also
\cite[Lemma 3]{Yagzhev} and \cite{Borel})
tells us that $F$ is an invertible polynomial map in case it is injective,
which is the case here.
\end{proof}

\begin{rem} When $d=2$ or $d=3$, Corollary \ref{allinj} gives a result of Wang
\cite[Theorem 1.2.2]{wang} and one of Yagzhev \cite[Theorem 1(ii)]{Yagzhev},
respectively. Corollary \ref{allinj} also generalizes
\cite[Theorem 2.2.1, Corollary 2.2.2]{Sun}.
\end{rem}

\begin{rem}
Now you might think that for Theorem \ref{lineinj},
the condition that there are $d-1$ collinear $\alpha_i$'s with
the additive property therein is weaker than
a similar property for $s$ $\alpha_i$'s, where $s \in \N$
is arbitrary. This is however not the case.
\end{rem}

\begin{thm} \label{lineadd}
Let $F: \C^n \rightarrow \C^m$ be a polynomial map of degree $\le d$ and
$\beta, \gamma \in \C^n$. Then the following statements are equivalent.
\begin{enumerate}

\item[(1)]
There exists $\lambda_1 ,\lambda_2, \ldots, \lambda_{d-1} \in \C$ satisfying
$\sum_{i\in I} \lambda_i \ne 0$ for all nonempty $I \subseteq \{1,2,\ldots,d-1\}$,
such that
$$
\sum_{i=1}^{d-1} \lambda_i \cdot \subss{\jac F}{X}{\alpha_i} \cdot \gamma \ne 0
$$
for all $\alpha_i \in \{\beta + \mu \gamma \mid \mu \in \C\}$.

\item[(2)] $F|_{\beta + \C \gamma}$ is linearly rectifiable
(in particular injective), i.e.\@ there exists a vector $v \in \C^m$ such that
\begin{equation} \label{lineaddv}
\sum_{j=1}^m v_j \cdot\dder{}{T} \big(F_j(\beta+T \gamma)\big) = 1
\end{equation}

\item[(3)]
For all $s \in \N$,
$$
\sum_{i=1}^s \lambda_i\cdot\subss{\jac F}{X}{\alpha_i} \cdot \gamma \ne 0
$$
for all $\lambda_i \in \C$ such that $\lambda_1 + \lambda_2 + \cdots +
\lambda_s \ne 0$, and all $\alpha_i \in \{\beta + \mu \gamma \mid \mu \in \C\}$.

\end{enumerate}
\end{thm}

\begin{proof}
Since (3) $\Rightarrow$ (1) is trivial, only two implications remain.
\begin{description}

\item[(2) \imp (3)] Assume that (2) is satisfied. Take $s \in \N$, $\lambda_1, \lambda_2,
\ldots, \lambda_s \in \C$ such that $\lambda_1 + \lambda_2 + \cdots + \lambda_s \ne 0$,
and $\alpha_i \in \{\beta + \mu \gamma \mid \mu \in \C\}$.
Each $\alpha_i$ is of the form $\alpha_i = \beta + r_i \gamma$ for some $r_i \in \C$.
By the chain rule,
\begin{align*}
v\tp \cdot \bigg( \sum_{i=1}^s \lambda_i\cdot
    \subss{\jac F}{X}{\alpha_i} \cdot \gamma \bigg)
&=  \sum_{i=1}^s \lambda_i\cdot\bigg(\sum_{j=1}^m v_j \cdot
    \subss{\jac F_j}{X}{\beta+r_i\gamma} \cdot \gamma \bigg) \\
&=  \sum_{i=1}^s \lambda_i\cdot
    \bsubss{\bigg}{\sum_{j=1}^m v_j \dder{}{T} \big(F_j (\beta + T \gamma)\big)}{T}{r_i} \\
&=  \sum_{i=1}^s \lambda_i \cdot \subs{1}{T}{r_i} = \sum_{i=1}^s \lambda_i \ne 0
\end{align*}
which gives (3).

\item [(1) \imp (2)]
Assume that (2) does not hold. We will derive a contradiction by showing
that (1) does not hold either.

Since $\deg_T \dder{}{T} F_j(\beta + T \gamma) \le d-1$ for all $j$,
the $\C$-space $U$ that is generated by
$$
\dder{}{T} F_1(\beta + T \gamma), \dder{}{T} F_2(\beta + T \gamma),
\ldots, \dder{}{T} F_m(\beta + T \gamma)
$$
has dimension $s \le d-1$, for $1 \notin U$.
Take a basis of $U$ of monic $u_1, u_2, \allowbreak \ldots,
\allowbreak u_s \in \C[T]$ such that $0 < \deg u_1 < \deg u_2
< \cdots < \deg u_s < d$. Write $u_{ji}$ for
the coefficient of $T^i$ of $u_j$.

Next, define $p_i$ for $i = 0,1,\ldots,d-1$
as follows.
$$
p_i := \left\{ \begin{array}{ll}
       -\sum_{k=0}^{i-1} p_k u_{jk} & \mbox{if $u_j$ has degree $i$.} \\
       \lambda_1 + \lambda_2 + \cdots + \lambda_{d-1} & \mbox{if no $u_j$ has degree $i$,}
\end{array} \right.
$$
Set $P := \sum_{k=1}^{d-1} \frac{p_k}{k!} T^k$ and choose $r_i$ as in Lemma \ref{addexp1}
for all $i$. Looking at the term expansion of $u_j$, we see that
$$
P\Big(\dder{}{T}\Big)u_j =
\sum_{k=0}^{\infty} \frac{p_k}{k!} \cdot \sum_{l=0}^{\infty} \frac{(k+l)!}{l!} u_{jk} T^l
$$
whence for $i = \deg u_j$
\begin{align*}
\bsubss{\bigg}{P\Big(\dder{}{T}\Big) u_j}{T}{0} &=
\sum_{k=0}^{\infty} p_k u_{jk} = p_{i} + \sum_{k=0}^{i-1} p_k u_{jk} = 0
\intertext{and similarly for each $i$}
\bsubss{\bigg}{\exp\Big(r_i \dder{}{T}\Big)u_j}{T}{0} &=
\sum_{k=0}^{\infty} r_i^k u_{jk} = u_j(r_i) = \bsubs{\big}{u_j}{T}{r_i}
\end{align*}
follow for all $j$.

By Lemma \ref{addexp1}, $P - \sum_{i=1}^{d-1} \lambda_i \exp(r_i T)$
is divisible by $T^d$. Since $\deg u_j < d$ for all $j$,
$$
0 = \bsubss{\bigg}{P\Big(\dder{}{T}\Big)u_j}{T}{0}
= \sum_{i=1}^{d-1} \lambda_i \cdot \bsubss{\bigg}{\exp\Big(r_i \dder{}{T}\Big)u_j}{T}{0}
= \sum_{i=1}^{d-1} \lambda_i \bsubs{\big}{u_j}{T}{r_i}
$$
Since $\dder{}{T} F_j(\beta + T \gamma)$ is a $\C$-linear combination of
$u_1, u_2, \ldots, u_s$ for all $j$, we have
\begin{align*}
0 &= \sum_{i=1}^{d-1} \lambda_i \cdot
     \bsubss{\Big}{\jac_T \big(F(\beta + T \gamma)\big)}{T}{r_i} \\
  &= \sum_{i=1}^{d-1} \lambda_i \cdot
     \bsubss{\big}{\subss{\jac F}{X}{\beta + T \gamma} \cdot \gamma}{T}{r_i} \\
  &= \sum_{i=1}^{d-1} \lambda_i \cdot \bsubss{\big}{\jac F}{X}{\beta + r_i \gamma}
     \cdot \gamma
\end{align*}
which is a contradiction. \qedhere

\end{description}
\end{proof}

\begin{rem}
For the map $F = (X_1 + (X_2 + X_1^2)^2, X_2 + X_1^2)$, only images of
lines parallel to the $X_2$-axis are linearly rectifiable. But all images of
lines are linearly rectifiable when
$F = (X_1 + (X_2 + X_1^2)^2 - (X_3 + X_1^2)^2, X_2 + X_1^2, X_3 + X_1^2)$
or any other invertible cubic map over $\C$. This follows from the proposition below.
\end{rem}

\begin{prop} \label{cubicrectif}
Let $F: \C^n \rightarrow \C^m$ be a polynomial map of degree $\le 3$, and
$\beta, \gamma \in \C^n$ such that $\gamma \ne 0$. If $F|_{\beta + \C \gamma}$
is injective and $(\jac F)|_{X = \alpha} \cdot \gamma \ne 0$ for all
$\alpha \in \{\beta + \mu \gamma \mid \mu \in \C\}$, then $F|_{\beta + \C \gamma}$
is linearly rectifiable, i.e.\@ there exists a $v \in \C^m$ such that (\ref{lineaddv}) holds.
\end{prop}

\begin{proof}
Assume $F|_{\beta + \C \gamma}$ is not linearly rectifiable.
Then there exist monic $u_1, u_2 \in \C[T]$ such that $\deg u_i = i$ and
for all $j$, $\dder{}{T} F_j(\beta + T \gamma)$ is linearly dependent over $\C$ of
$u_1$ and $u_2$. If the constant term $u_{10}$ of $u_1$ is nonzero, then $u_{10}$
will become zero after replacing $\beta$ by $\beta - u_{10}\gamma$ and adapting
$u_1$ and $u_2$ accordingly.
So assume $u_{10} = 0$ and let $u_{20}$ be the constant term of
$u_2$. By taking the integral of
$u_1$ and $u_2$ from $T = - \sqrt{-3u_{20}}$ to $T = + \sqrt{-3u_{20}}$, we see that
 $F(\beta - \sqrt{-3u_{20}} \gamma) = F(\beta + \sqrt{-3u_{20}} \gamma)$,
thus either $F|_{\beta + \C \gamma}$ is not injective
or $u_{20} = 0$. If $u_{20} = 0$, then $(\jac F)|_{X = \beta} \cdot \gamma = 0$
because both $u_1$ and $u_2$ are divisible by $T$.
This completes the proof of Proposition \ref{cubicrectif}.
\end{proof}

\begin{cor}
Assume $F: \C^n \rightarrow \C^n$ is a polynomial map of degree $\le 3$
which safisfies the Keller condition $\det \jac F \in \C^{*}$.
Then $F$ is invertible, if and only if $F|_{L}$ is linearly rectifiable
for every affine line $L \subseteq \C^n$, if and only if
$\big(\subss{\jac F}{X}{\alpha}+\subss{\jac  F}{X}{\beta}\big)(\alpha-\beta)\neq 0$
for all $\alpha,\beta\in  \C^n$ with $\alpha \neq \beta$.
\end{cor}

\begin{proof}
By Proposition \ref{cubicrectif}, $F$ is invertible, if and only if $F|_{L}$
is linearly rectifiable for every affine line $L \subseteq \C^n$.
By Proposition \ref{lineadd}, the latter is equivalent to
$\big(\subss{\jac F}{X}{\alpha}+\subss{\jac  F}{X}{\beta}\big)(\alpha-\beta)\neq 0$
for all $\alpha,\beta\in  \C^n$ with $\alpha \neq \beta$, as desired.
\end{proof}

\begin{rem}
Notice that in the proof of Lemma \ref{addexp1}, we solve
$d-1$ equations in $d-1$ variables to obtain $r_1, r_2, \ldots, r_{d-1}$.
In case $\lambda_1 = \lambda_2 = \cdots = \lambda_{d-1}$, it suffices to
solve only one equation in only one variable to obtain $r_1, r_2, \ldots,
r_{d-1}$.

\begin{lemalt} \label{addexp2}
Let $P \in \C[[T]]$ with constant term $d-1$. Then there are
$r_1, r_2, \ldots, \allowbreak r_{d-1} \in \C$, which are roots of
a polynomial whose coefficients are polynomials in those of $P$, such that
$$
P - \sum_{i=1}^{d-1} \exp(r_i T)
$$
is divisible by $T^d$, where $\exp(T) = \sum_{j=0}^{\infty}
\frac{1}{j!} T^j$.
\end{lemalt}

\begin{proof}
Write
$$
P = \sum_{j=0}^{\infty} \frac{p_j}{j!} T^j
$$
Then we must find a solution $(Y_1, Y_2, \ldots, Y_{d-1}) =
(r_1, r_2, \ldots, r_{d-1})$ of
$$
\sum_{i=1}^{d-1} Y_i^j = p_j \quad (j = 0,1,\ldots,d-1)
$$
By Newton's identities for symmetric polynomials, there
exist a polynomial $f \in \C[T][X_1,X_2,\ldots,X_{d-1}]$ which is
injective as a function of $\C^{d-1}$ to $\C[T]$, such that
$$
f \left(\sum_{i=1}^{d-1} X_i, \sum_{i=1}^{d-1} X_i^2, \ldots,
  \sum_{i=1}^{d-1} X_i^{d-1}\right) = \prod_{i=1}^{d-1} (T + X_i)
$$
Notice that $g := f(p_1, \ldots, p_{d-1})$ is a monic polynomial
of degree $d-1$ in $T$. Hence we can decompose $g$ as
$$
g = \prod_{i=1}^{d-1} (T + r_i) = f\left(\sum_{i=1}^{d-1} r_i,
  \sum_{i=1}^{d-1} r_i^2, \ldots, \sum_{i=1}^{d-1} r_i^{d-1}\right)
$$
and the injectivity of $f$ gives the desired result.
\end{proof}
\end{rem}

\section{Additive properties of the Jacobian determinant}

\begin{prop} \label{quadr}
Let $F: \C^n \rightarrow \C^n$ be a quadratic polynomial map such that
$\det \jac F \in \C$. Then for all $s \in \N$,
$$
\det\bigg(\sum_{i=1}^s b_i \cdot \subss{\jac F}{X}{\alpha_i}\bigg)
= \det \bigg(\sum_{i=1}^s b_i \cdot \jac F \bigg)
= \bigg(\sum_{i=1}^s b_i\bigg)^n \cdot \det \jac F
$$
for all $\alpha_1, \alpha_2, \ldots, \alpha_s \in \C^n$ and all
$b_1, b_2, \ldots, b_s \in \C$.
\end{prop}

\begin{proof}
Since the entries of $\jac F$ are affinely linear, we have
$$
\sum_{i=1}^s b_i \cdot \subss{\jac F}{X}{\alpha_i}
= \sigma \cdot \bbsubss{\bigg}{\Big.}{\jac F}{\textstyle X}{\sigma^{-1}
  \sum_{i=1}^s b_i \alpha_i}
$$
for all $\alpha_1, \alpha_2, \ldots, \alpha_s \in \C^n$ and all
$b_1, b_2, \ldots, b_s \in \C$, in case $\sigma := \sum_{i=1}^s b_i \ne 0$.
Taking determinants on both sides, it follows from $\det \jac F \in \C$ that
$$
\det\bigg(\sum_{i=1}^s b_i \cdot \subss{\jac F}{X}{\alpha_i}\bigg)
= \det (\sigma \cdot \jac F) = \sigma^n \cdot \det \jac F
$$
when $\sigma \ne 0$, and by continuity also in case $\sigma = 0$, as desired.
\end{proof}

\begin{lem} \label{lm3.2}
Assume $f \in \C[X]$ has degree $\le d$. If $f$ vanishes on the set
$S := \{a \in \N^n \mid a_1 + a_2 + \cdots + a_n \le d \}$, then $f = 0$.
\end{lem}

\begin{proof}
Write $f = (f|_{X_n = 0}) + X_n \cdot (g|_{X_n=X_n-1})$. By induction on
$n$, $(f|_{X_n = 0}) = 0$. Furthermore, if $a \in S$ and
$a_n \ge 1$, then
$$
g(a_1,a_2,\ldots,a_{n-1},a_n-1) = (g|_{X_n=X_n-1})(a) =
\frac{f(a) - (f|_{X_n=0})(a)}{a_n} = 0
$$
thus by induction on $d$, $g = 0$. Hence $f=0$ as well.
\end{proof}

\begin{cor} \label{cor3.2}
Let $f \in \C[X]$ be a polynomial of degree $\le d$.
If $f(a) = 0$ for all $a \in \N^n$ such that
$\sum_{i=1}^n a_i = d$, then $\sum_{i=1}^n x_i - d \mid f$. If additionally $f$
is homogeneous, then $f = 0$.
\end{cor}

\begin{proof}
If we substitute $X_n = d - \sum_{i=1}^{n-1} X_i$ in $f$, then we
get a polynomial of degree $\le d$ which is zero on account of Lemma
\ref{lm3.2}. Hence $X_n = d - \sum_{i=1}^{n-1} X_i$ is a zero of $f
\in \C(X_1,X_2,\ldots,X_{n-1})[X_n]$ and $f$ is divisible over
$\C(X_1,X_2,\ldots,X_{n-1})$ by $\sum_{i=1}^n X_i - d$. By Gauss'
Lemma, $f$ is divisible over $\C[X]$ by $\sum_{i=1}^n X_i - d$,
which is only homogeneous if $d = 0$. Hence $f = 0$ when $f$ is
homogeneous.
\end{proof}

\begin{lem} \label{lm3.3}
Let $F: \C^n \rightarrow \C^m$ be a polynomial map and $P: \Mat_{m,n}(\C) \rightarrow \C$
be a polynomial of degree $\le d$ in the entries of its input matrix. Fix $\mu \in \C$
and assume that
$$
P\bigg(\sum_{i=1}^d \subss{\jac F}{X}{\alpha_i}\bigg) = \mu
$$
for all $\alpha_1, \alpha_2, \ldots, \alpha_d \in \C^n$. Then for all $s \in \N$
$$
P\bigg(\sum_{i=1}^s b_i \cdot \subss{\jac F}{X}{\alpha_i}\bigg) = \mu = P(d \jac F)
$$
for all $\alpha_1, \alpha_2, \ldots, \alpha_s \in \C^n$ and all
$b_1, b_2, \ldots, b_s \in \C$ such that $\sum_{i=1}^s b_i = d$.

If additionally $P$ is homogeneous, then
$$
P\bigg(\sum_{i=1}^s b_i\cdot \subss{\jac F}{X}{\alpha_i}\bigg) =
\bigg(\frac1d \sum_{i=1}^s b_i\bigg)^{\deg P} \mu =
\bigg(\sum_{i=1}^s b_i\bigg)^{\deg P} P(\jac F)
$$
for all $\alpha_1, \alpha_2, \ldots, \alpha_s \in \C^n$ and all
$b_1, b_2, \ldots, b_s \in \C$.
\end{lem}

\begin{proof}
Since $P\big(\sum_{i=1}^d \subss{\jac F}{X}{\alpha_i}\big) = \mu$ is constant,
$$
\mu = P\bigg(\sum_{i=1}^d \subss{\jac F}{X}{\alpha_i}\bigg) = P(d \jac F)
$$
for all $\alpha_i \in \C^n$. Take $\alpha_1, \alpha_2, \ldots, \alpha_s \in \C^n$ and let
$$
f(Y_1,Y_2,\ldots,Y_s) := P\bigg(\sum_{i=1}^s Y_i\cdot\subss{\jac F}{X}{\alpha_i}\bigg)
- \mu
$$
Then $\deg f(Y_1,Y_2,\ldots,Y_s) \le d$, and for all
$b \in \N^s$ such that $\sum_{i=1}^s b_i = d$, we have
\begin{align*}
f(b) &= P\bigg(\sum_{i=1}^s b_i\cdot\subss{\jac F}{X}{\alpha_i}\bigg) - \mu \\
&= P\bigg(\sum_{i=1}^s \sum_{j=1}^{b_i} \subss{\jac F}{X}{\alpha_i}\bigg) - \mu = 0
\end{align*}
By Corollary \ref{cor3.2}, $\sum_{i=1}^s Y_i - d \;\big|\;
f(Y_1,Y_2,\ldots,Y_s) - \mu$, whence
$$
0 \;\bigg|\; P\bigg(\sum_{i=1}^s b_i\cdot \subss{\jac F}{X}{\alpha_i}\bigg) - \mu
$$
for all $b \in \C^s$ such that $\sum_{i=1}^s b_i = d$. This gives the first assertion of
Lemma \ref{lm3.3}.

Assume $P$ is homogeneous. Then
\begin{align*}
g(Y_1, Y_2, \ldots, Y_s)
&:= P\bigg(\sum_{i=1}^s Y_i\cdot \subss{\jac F}{X}{\alpha_i}\bigg) -
\bigg(\frac1d\sum_{i=1}^s Y_i\bigg)^{\deg P} \mu \\
&\;= P\bigg(\sum_{i=1}^s Y_i\cdot \subss{\jac F}{X}{\alpha_i}\bigg) -
\bigg(\sum_{i=1}^s Y_i\bigg)^{\deg P} P(\jac F)
\end{align*}
is homogeneous as well. Since $g$ vanishes on $b$ for all
$b \in \N^s$ such that $\sum_{i=1}^s b_i = d$, we obtain from Corollary \ref{cor3.2}
that $g = 0$, which completes the proof of Lemma \ref{lm3.3}.
\end{proof}

\begin{thm} \label{th2}
Let $m \ge n$ and $F: \C^n \rightarrow \C^n$ be a polynomial map
such that for a fixed $\mu \in \C$, we have
$$
\det\bigg(\sum_{i=1}^m \subss{\jac F}{X}{\alpha_i}\bigg) = \mu
$$
for all $\alpha_1, \alpha_2, \ldots, \alpha_m \in \C^n$. Then $\mu = \det (m \jac F)
= m^n \det(\jac F)$
and for all $s \in \N$
$$
\det\bigg(\sum_{i=1}^s b_i\cdot \subss{\jac F}{X}{\alpha_i}\bigg)
= \bigg(\frac1m \sum_{i=1}^s b_i\bigg)^n \mu
= \bigg(\sum_{i=1}^s b_i\bigg)^n \det (\jac F)
$$
for all $\alpha_1, \alpha_2, \ldots, \alpha_s \in \C^n$ and all
$b_1, b_2, \ldots, b_s \in \C$. Furthermore, $F$ is an invertible polynomial map
in case $\det \jac F \ne 0$.
\end{thm}

\begin{proof}
To obtain the first assertion, take $P = \det$, $d = m$ and $m = n$ in Lemma
\ref{lm3.3}. By taking $s = \deg F - 1$ and $b_i = 1$ for all $i$ in this assertion,
it follows from Corollary \ref{allinj} that $F$ is an invertible polynomial map in
case $\det \jac F \ne 0$.
\end{proof}

\begin{thm}
Assume $H: \C^n \rightarrow \C^n$ is a polynomial map and define
$$
M(\alpha_1,\alpha_2,\ldots,\alpha_s) := \subss{\jac H}{X}{\alpha_1} +
\subss{\jac H}{X}{\alpha_2} + \cdots + \subss{\jac H}{X}{\alpha_s}
$$
If for some $m \ge d$, the sum of the principal
minors of size $d$ of $M(\alpha_1,\alpha_2,\ldots,\alpha_m)$ is zero
for all $\alpha_i \in \C^n$, then for all $s \in \N$, the sum of the principal
minors of size $d$ of
\begin{equation}
b_1 \cdot \subss{\jac H}{X}{\alpha_1} +
b_2 \cdot \subss{\jac H}{X}{\alpha_2} + \cdots + b_s \cdot \subss{\jac H}{X}{\alpha_s} \label{bsum}
\end{equation}
is zero as well, for all $b_i \in \C$ and all $\alpha_i \in \C^n$.
If for some $m \ge d$, the trace of
$M(\alpha_1,\alpha_2,\ldots,\allowbreak \alpha_m)^d$ is zero
for all $\alpha_i \in \C^n$, then for all $s \in \N$, the trace of the $d$-the power
of (\ref{bsum}) is
zero as well, for all $b_i \in \C$ and all $\alpha_i \in \C^n$.
\end{thm}

\begin{proof}
Take for $P$ the sum of the principal minors of size $m$ or the trace
of the $m$-th power, respectively. By Lemma \ref{lm3.3}, $P((\ref{bsum}))$
is divisible by $\mu := P(m \jac H) =
P(M(\alpha_1,\alpha_2,\ldots,\allowbreak \alpha_m)) = 0$.
\end{proof}

\begin{rem}
Let $F=X+H$ such that the Jacobian matrix $\jac H$ is additive-nilpotent. Then
for all $m \in \N$,
$\sum^m_{i=1}\subss{\jac F}{X}{\alpha_i}$ is invertible for all $\alpha_i\in
\mathbb{C}^n$. We shall show below that the converse holds when
$H$ does not have linear terms. But
the converse is not true in general. For example, let $F(X)=X+H$,
where $H=(-X_1+X_2,X_1-X_2+X_2^2)$. Then
$$
\jac H=\begin{pmatrix}
 -1&1\\1&2X_2-1
\end{pmatrix} \qquad \mbox{and} \qquad \jac F=\begin{pmatrix}
 0&1\\1&2X_2
\end{pmatrix}
$$
such that $\jac H$ is not even nilpotent and
$\sum^{2}_{i=1}\subss{\jac F}{X}{\alpha_i}$ is invertible for all $\alpha_1, \alpha_2\in
\mathbb{C}^2$.
\end{rem}

\begin{prop}\label{dpr3}
Assume $F: \mathbb{C}^n\rightarrow \mathbb{C}^n$ is a polynomial map of the
form $F = L + H$, such that $L$ is invertible and $\deg L = 1$.
Then for all $s \in \N$, all $b_i \in \C$, and all
$\alpha_i \in \C^n$, the following statements are equivalent to each other.
\begin{enumerate}
\item[(1)] For all $\mu \in \C$, we have
$$
\det\bigg(\mu \cdot \jac L + \sum_{i=1}^s b_i \cdot
\subss{\jac F}{X}{\alpha_i}\bigg) =
\bigg(\mu + \sum_{i=1}^s b_i\bigg)^n \cdot \det (\jac L)
$$
\item[(2)] $\sum_{i=1}^s b_i
\cdot \bsubss{\big}{\jac (L^{-1} \circ H)}{X}{\alpha_i}$ is nilpotent.
\end{enumerate}
\end{prop}

\begin{proof}
Assume (1). Since the equality of (1) holds for all $\mu \in \C$, we obtain
$$
\det\bigg(T \cdot \jac L + \sum_{i=1}^s b_i \cdot
\subss{\jac F}{X}{\alpha_i}\bigg) =
\bigg(T + \sum_{i=1}^s b_i\bigg)^n \cdot \det (\jac L)
$$
which is equivalent to
$$
\det\bigg(\bigg(T - \sum_{i=1}^s b_i\bigg) \cdot \jac L + \sum_{i=1}^s b_i \cdot
\subss{\jac L + \jac H}{X}{\alpha_i}\bigg) = T^n \cdot \det (\jac L)
$$
and
$$
\det\bigg(T \cdot \jac L + \sum_{i=1}^s b_i \cdot
\subss{\jac H}{X}{\alpha_i}\bigg) = T^n \cdot \det (\jac L)
$$
By dividing both sides by $\det (\jac L)$, we obtain
$$
\det\bigg(T + \sum_{i=1}^s b_i \cdot
(\jac L)^{-1} \cdot \subss{\jac H}{X}{\alpha_i}\bigg) = T^n
$$
which implies (2). The converse is similar.
\end{proof}

\noindent
Let $F=X+H$ such that $\jac H$ is additive-nilpotent. Then
$\sum_{i=1}^{m}\subss{\jac\wtilde{F}}{X}{\alpha_i}$ is
invertible for all $\alpha_i\in \mathbb{C}^n$ and all positive integers
$m$, where $\wtilde{F}=L_1\circ F\circ L_2$ for invertible linear maps $L_1$ and $L_2$.
We next prove that the converse holds.

\begin{thm}\label{dth3}
For a polynomial map $F: \mathbb{C}^n\rightarrow \mathbb{C}^n$ the following statements are equivalent.
\begin{enumerate}
\item[(1)] $\sum^{n}_{i=1}\subss{\jac F}{X}{\alpha_i}$ is invertible for all
$\alpha_i\in \mathbb{C}^n$;
\item[(2)] $F=L\circ(X+H)$, where $H$ has no linear terms,
the linear part $L$ of $F$ is invertible
and $\jac H$ is additive-nilpotent;
 \item[(3)] $F=(X+H)\circ L$, where $H$ has no linear terms,
the linear part $L$ of $F$ is invertible
and $\jac H$ is additive-nilpotent;
 \item[(4)] $F=L_1\circ(X+H)\circ L_2$, where $L_1$ and $L_2$ are
invertible maps of degree one and $\jac H$ is additive-nilpotent.
 \end{enumerate}
\end{thm}

\begin{proof}
Since (3) $\Rightarrow$ (4) is trivial, the following three implications remain
to be proved.
\begin{description}

\item[(4) \imp (1)] Assume (4). Since $\jac H$ is additive-nilpotent, (1) holds with
$X + H$ instead of $F$. Since (1) is not affected by compositions with
translations and invertible linear maps, and $F$ can be obtained from
$X + H$ in that manner, (1) follows.

\item[(1) \imp (2)] Assume (1). By the fundamental theorem of algebra,
the determinant of $\sum^{n}_{i=1}\subss{\jac F}{X}{\alpha_i}$ is a nonzero constant
which does not depend on $\alpha_1,\alpha_2,\ldots,\alpha_n$. Let $L$
be the linear part of $F$ . By theorem
\ref{th2}, we obtain that $\det \jac F = \det \subss{\jac F}{X}{0} = \det \jac L$
and that (1) of proposition \ref{dpr3} holds for all $s \in \N$, all $b_i \in \C$,
and all $\alpha_i \in \C^n$. Hence the Jacobian of
$H := L^{-1} \circ (F-L)$
is additive-nilpotent on account of proposition \ref{dpr3}, which gives the
desired result.

\item[(2) \imp (3)] This follows from the fact that
$F=L\circ (X+H)= \big(X+(L\circ H\circ L^{-1})\big)\circ L$ and the Jacobian of
$L\circ H\circ L^{-1}$ is also additive-nilpotent. \qedhere

\end{description}
\end{proof}

\begin{rem}
A polynomial map $F=(F_1,\ldots,F_n)$ is called triangular if its Jacobian matrix
is triangular, i.e.\@ either above or below the main diagonal, all entries of $\jac F$
are zero. The Jacobian matrix $\jac F$ of a
triangular invertible polynomial map $F$ can only have nonzero constants on the main diagonal, and thus for all invertible linear maps $L_1$ and $L_2$, $\sum_{i=1}^{n}
\bbsubss{\big}{{}}{\jac(L_1\circ F\circ L_2)}{X}{\alpha_i}$ is invertible for all $\alpha_i\in \mathbb{C}^n$. However, a polynomial map satisfying the conditions of Theorem \ref{dth3} is not necessarily a composition of a triangular map and two linear maps.
Indeed, in \cite{Meisters}, it was shown that in dimension 5 and up,
Keller maps $X + H$ with $H$ quadratic homogeneous do not necessarily have the property that
$\jac H$ is strongly nilpotent. But on account of Proposition \ref{quadr}, such maps satisfy
property (1) of Theorem \ref{dth3}.

In \cite{homokema}, all those maps such that
$\jac H$ is not strongly nilpotent are determined in dimension 5. $H$ is either of the form
$$
H = L^{-1} \circ \left(\left(\begin{array}{c}
0 \\ \lambda X_1^2 \\ X_2 X_4 \\ X_1 X_3 - X_2 X_5 \\ X_1 X_4
\end{array} \right) + \left( \begin{array}{c}
0 \\ 0 \\ p(X_1,X_2) \\ q(X_1,X_2) \\ r(X_1,X_2)
\end{array} \right) \right) \circ L
$$
where $\lambda \in \{0,1\}$, $L$ is linear and $p,q,r \in \C[x_1,x_2]$, or
of the form
$$
H = L^{-1} \circ \left(\left(\begin{array}{c}
0 \\ X_1 X_3 \\ X_2^2 - X_1 X_4 \\ 2 X_2 X_3 - X_1 X_5 \\ X_3^2
\end{array} \right) + \left( \begin{array}{c}
0 \\ \lambda_2 X_1^2 \\ \lambda_3 X_1^2 \\ \lambda_4 X_1^2 \\ \lambda_5 X_1^2
\end{array} \right) \right) \circ L
$$
where $L$ is linear and $\lambda_i \in \C$. One can show that in both cases, the columns
of $\jac (L \circ H \circ L^{-1})$ are linearly independent over $\C$,
something that cannot be counteracted with compositions with invertible linear maps.
Hence the columns of $\jac(L_1\circ H\circ L_2)$ are linearly independent over $\C$
for all invertible maps  $L_i$. $\jac(L_1\circ H\circ L_2)$ is exactly the
linear part of $\jac(L_1\circ F\circ L_2)$,
thus $\jac(L_1\circ F\circ L_2)$ can only be triangular if its main diagonal is not constant
on one of its ends. This is however not possible since $L_1\circ F\circ L_2$
is invertible.
\end{rem}


\end{document}